\documentclass[a4paper,12pt]{article}
\usepackage{stmaryrd}
\usepackage{mathrsfs}
\usepackage{amsmath,amssymb,amsthm,graphics,latexsym,amsfonts}
\usepackage{fancyhdr}
\usepackage{CJK}
\usepackage{graphics}
\usepackage{color}

\newtheorem{theorem}{Theorem}[section]
\newtheorem{lemma}[theorem]{Lemma}
\newtheorem{proposition}[theorem]{Proposition}
\newtheorem{corollary}[theorem]{Corollary}
\theoremstyle{plain} \CJKtilde

\usepackage{indentfirst}
\begin{document}

\title{\Large Total [1,2]-domination in graphs\thanks{This work was supported by National Natural Science Foundation of
China (No. 11001269, No. 11161046)}}
\author{ Xuezheng Lv\\
\small Department of Mathematics, Renmin University of China,\\
\vspace{3mm}\small  Beijing 100872, P.R. China\\
{Baoyindureng Wu \thanks{Corresponding
author.\newline{\hspace*{5mm}} E-mail address:
xzlv@ruc.edu.cn (X. Lv); wubaoyin@hotmail.com(B. Wu)} }\\
\small  College of Mathematics and System Sciences, Xinjiang
University \\ \small  Urumqi, Xinjiang 830046, P.R. China \\
}

\date{}

\maketitle

{\small \noindent{\bfseries Abstract}: A subset $S\subseteq V$ in a
graph $G=(V,E)$ is a total $[1,2]$-set if, for every vertex $v\in
V$, $1\leq |N(v)\cap S|\leq 2$. The minimum cardinality of a total
$[1,2]$-set of $G$ is called the total $[1,2]$-domination number,
denoted by $\gamma_{t[1,2]}(G)$.

We establish two sharp upper bounds on the total [1,2]-domination
number of a graph $G$ in terms of its order and minimum degree, and
characterize the corresponding extremal graphs achieving these
bounds. Moreover, we give some sufficient conditions for a graph
without total $[1,2]$-set and for a graph with the same total
$[1,2]$-domination number, $[1,2]$-domination number and domination
number.
}

\section{Introduction}
We consider undirected finite simple graphs only, and refer to
\cite{W} for undefined notations and terminologies. Let $G=(V(G),
E(G))$ be a graph of order $n=|V(G)|$ and size $m=|E(G)|$. For a
vertex $v\in V(G)$, {\it the neighborhood} $N_G(v)$ of $v$ is the set
of vertices adjacent to $v$ in $G$, and the {\it closed
neighborhood} $N_G[v]$ of $v$ is $N_G(v)\cup \{v\}$; the {\it
degree} $d_G(v)$ of $v$ is the number of edges incident with $v$ in
$G$. Since $G$ is simple, $d_G(v)=|N_G(v)|$. If there is no
confusion, we simply write $N(v)$, $N[v]$ and $d(v)$ instead of
$N_G(v)$, $N_G[v]$ and $d_G(v)$. The minimum degree and maximum degree of
a vertex in a graph $G$ are denoted by $\delta(G)$ and $\Delta(G)$,
respectively.  A vertex of degree 1 is called a {\it leaf} and the
vertex adjacent to a leaf is called a {\it support vertex}.

Let $S\subseteq V(G)$ and $v\in S$. The open and closed neighborhood
of $S$ are $N(S)=\cup_{v\in S}N(v)$ and $N[S]=\cup_{v\in S}N[v]$,
respectively. The {\it $S$-private neighborhood }of $v$, denoted by
$pri_G(v, S)$ or simply $pri(v,S)$, consists of all vertices in
$N_G[v]$ but not in $N_G[S\setminus \{v\}]$; that is, $pri(v,
S)=N_G[v]\setminus N_G[S\setminus \{v\}]$. Thus, if $u\in pri(v,
S)$, then $N[u]\cap S=\{v\}$. For $S'\subseteq S$, denote $pri(S',
S)=\bigcup_{u\in S'}pri(u, S)$. The subgraph obtained by deleting
all vertices in $S$ is denoted by $G-S$. The symbol $G[S]$ denotes
the subgraph of $G$ induced by $S$. So, $G-S$ is the subgraph of $G$ induced by
$V(G)\setminus S$.

Let $X$ and $Y$ be two sets of vertices of a graph $G$. We denote by
$E[X, Y]$ the set of edges of $G$ with one end in $X$ and the other
end in $Y$.

A set $S\subseteq V(G)$ is called a {\it dominating set} of a graph
$G$ if $N[S]=V(G)$, that is, $N(v)\cap S\neq \emptyset$ for every
vertex $v\in V(G)\setminus S$. The minimum cardinality of a
dominating set in a graph $G$ is called the {\it domination number} of $G$,
and is denoted by $\gamma(G)$. A dominating set $S$ of $G$ is called
a {\it total dominating set} if for every vertex $v\in V(G)$,
$N(v)\cap S\neq \emptyset$. The {\it total domination number} of
$G$, denoted by $\gamma_{t}(G)$, is the minimum cardinality of a
total dominating set of $G$. An independent set of $G$ is a set of
mutually independent vertices.

The notion of $[1,2]$-domination was investigated by Dejter
\cite{dej} and more recently, by Chellali et al. \cite{che}. A
dominating set $S$ of $G$ is called a {\it $[1,2]$-set} if $1\leq
|N(v)\cap S|\leq 2$ holds for every vertex $v\in V(G)\setminus S$,
that is, $S$ is a dominating set and every vertex in $V\setminus S$
is adjacent to no more than two vertices of $S$. The {\it
$[1,2]$-domination number} of $G$, denoted by $\gamma_{[1,2]}(G)$,
is the minimum cardinality of a $[1,2]$-set of $G$. Since the vertex
set $V(G)$ itself is a $[1,2]$-set of a graph $G$, any graph has a
$[1,2]$-set. There exists an infinite family of graphs $G$ whose
$[1, 2]$-domination number are equal to their orders \cite{che, wu}.
So, one of the fundamental problem on $[1, 2]$-domination of graphs
is that for which graphs $G$ of order $n$ is $\gamma_{[1,2]}(G)=n$ ?
The problem, whether there is a simple, polynomial test for deciding
if $\gamma_{[1,2]}(G)<n$, is still open. A number of open problems
in \cite{che} regarding [1,2]-domination were solved by Yang and Wu
\cite{wu}.

An analogue notion of [1,2]-domination for the total domination of a
graph was also introduced in \cite{che}. A total dominating set $S$
of $G$ is called a {\it total $[1,2]$-set} if for every vertex $v\in
V(G)$, $1\leq |N(v)\cap S|\leq 2$. The {\it total $[1,2]$-domination
number} of $G$, denoted by $\gamma_{t[1,2]}(G)$, is the minimum
cardinality of a total $[1,2]$-set of $G$. However, we will see that
there exist an infinite number of graphs with no total $[1,2]$-set,
even among trees. For convenience, if there is no total $[1,2]$-set
in a graph $G$, we denote $\gamma_{t[1,2]}(G)=+\infty$. By this
convention, it is trivial to see that  for any graph $G$,
$$\gamma(G)\leq \gamma_{t}(G)\leq \gamma_{t[1,2]}(G) \ and\
\gamma(G)\leq \gamma_{[1,2]}(G)\leq \gamma_{t[1,2]}(G).$$

We call $S\subseteq V(G)$ a {\it $\gamma(G)$-set}, if $S$ is a
dominating set with $|S|=\gamma(G)$. Similarly,
$\gamma_{[1,2]}(G)$-set, $\gamma_{t}(G)$-set,
$\gamma_{t[1,2]}(G)$-set can be defined. One of the remaining
problems in \cite{che} was stated as follows.

\vspace{3mm}\noindent{\bf  Question (\cite{che})}. What can you say
about total $[1,2]$-set, and the corresponding total
$[1,2]$-domination number $\gamma_{t[1,2]}(G)$ ?

\vspace{3mm} In this paper, we give two sharp upper bounds on the
total $[1,2]$-domination number of a graph in terms of its order and
an additional condition that the minimum degree at least 1 or at
least 2. Moreover, we give some sufficient conditions for a graph
without total $[1,2]$-set and for a graph with the same total
$[1,2]$-domination number, $[1,2]$-domination number and domination
number.

\section{ Upper bounds for $\gamma_{t[1,2]}(G)$}

As usual, the path, cycle, and complete graph of order $n$ are
denoted by $P_n$, $C_n$ and $K_n$, respectively. The $k$-corona
$G\circ P_k$ is the graph obtained from $G$ by attaching a pendant
path of length $k-1$ to each vertex $v\in V(G)$; the double
$k$-corona $G\circ 2P_k$ is the graph obtained from $G$ by attaching
two pendant paths of length $k$ to each vertex $v\in V(G)$, see
Fig. 1 for an illustration.
\vspace{5mm}
\setlength{\unitlength}{6mm}\newsavebox{\Mx} \savebox{\Mx}
{\begin{picture}(0,0) \put(0,0){\circle*{0.2}}
\put(1,0){\circle*{0.2}}\put(1,1){\circle*{0.2}}
\put(0,1){\circle*{0.2}}
\put(2,1){\circle*{0.2}}\put(2,0){\circle*{0.2}}

\qbezier{(0,0)(1,0)(2,0)} \qbezier{(0,0)(0,0.5)(0,1)}
\qbezier{(0,1)(1,1)(2,1)} \qbezier{(1,0)(1,0.5)(1,1)}
\qbezier{(0,1)(0.5,0.5)(1,0)} \put(0.8,-1){$G$}
\end{picture}}

\setlength{\unitlength}{6mm}\newsavebox{\My} \savebox{\My}
{\begin{picture}(0,0) \put(0,0){\circle*{0.2}}
\put(1,0){\circle*{0.2}}\put(1,1){\circle*{0.2}}
\put(0,1){\circle*{0.2}}
\put(2,1){\circle*{0.2}}\put(2,0){\circle*{0.2}}
\put(0,2){\circle*{0.2}} \put(0,3){\circle*{0.2}}
\put(0,-1){\circle*{0.2}}\put(0,-2){\circle*{0.2}}
\put(1,2){\circle*{0.2}} \put(1,3){\circle*{0.2}}
\put(1,-1){\circle*{0.2}}\put(1,-2){\circle*{0.2}}
\put(3,0){\circle*{0.2}} \put(4,0){\circle*{0.2}}
\put(3,1){\circle*{0.2}}\put(4,1){\circle*{0.2}}

\qbezier{(0,0)(1,0)(2,0)} \qbezier{(0,0)(0,0.5)(0,1)}
\qbezier{(0,1)(1,1)(2,1)} \qbezier{(1,0)(1,0.5)(1,1)}
\qbezier{(0,1)(0.5,0.5)(1,0)}\qbezier{(0,0)(0,-1)(0,-2)}
\qbezier{(0,1)(0,2)(0,3)} \qbezier{(1,0)(1,-1)(1,-2)}
\qbezier{(1,1)(1,2)(1,3)} \qbezier{(2,0)(3,0)(4,0)}
\qbezier{(2,1)(3,1)(4,1)}

\put(0.8,-3){$G\circ P_2$}
\end{picture}}

\setlength{\unitlength}{6mm}\newsavebox{\Mz} \savebox{\Mz}
{\begin{picture}(0,0) \put(0,0){\circle*{0.2}}
\put(1,0){\circle*{0.2}}\put(1,1){\circle*{0.2}}
\put(0,1){\circle*{0.2}}
\put(2,1){\circle*{0.2}}\put(2,0){\circle*{0.2}}
\put(0,2){\circle*{0.2}} \put(0,3){\circle*{0.2}}
\put(0,-1){\circle*{0.2}}\put(0,-2){\circle*{0.2}}
\put(1,2){\circle*{0.2}} \put(1,3){\circle*{0.2}}
\put(1,-1){\circle*{0.2}}\put(1,-2){\circle*{0.2}}
\put(3,0){\circle*{0.2}} \put(4,0){\circle*{0.2}}
\put(3,1){\circle*{0.2}}\put(4,1){\circle*{0.2}}

\put(-1.2,2.8){\circle*{0.2}} \put(-0.6,1.8){\circle*{0.2}}
\put(-1.2,-1.8){\circle*{0.2}}\put(-0.6,-0.8){\circle*{0.2}}
\put(1.6,1.8){\circle*{0.2}} \put(2.2,2.8){\circle*{0.2}}
\put(1.6,-0.8){\circle*{0.2}}\put(2.2,-1.8){\circle*{0.2}}
\put(2.8,-0.6){\circle*{0.2}} \put(3.8,-1.2){\circle*{0.2}}
\put(2.8,1.6){\circle*{0.2}}\put(3.8,2.2){\circle*{0.2}}

\qbezier{(0,0)(1,0)(2,0)} \qbezier{(0,0)(0,0.5)(0,1)}
\qbezier{(0,1)(1,1)(2,1)} \qbezier{(1,0)(1,0.5)(1,1)}
\qbezier{(0,1)(0.5,0.5)(1,0)}\qbezier{(0,0)(0,-1)(0,-2)}
\qbezier{(0,1)(0,2)(0,3)} \qbezier{(1,0)(1,-1)(1,-2)}
\qbezier{(1,1)(1,2)(1,3)} \qbezier{(2,0)(3,0)(4,0)}
\qbezier{(2,1)(3,1)(4,1)}

\qbezier{(0,0)(-0.6,-0.8)(-1.2,-1.8)}
\qbezier{(0,1)(-0.6,1.8)(-1.2,2.8)}
\qbezier{(1,1)(1.6,1.8)(2.2,2.8)}
\qbezier{(1,0)(1.6,-0.8)(2.2,-1.8)}
\qbezier{(2,1)(2.8,1.6)(3.8,2.2)}
\qbezier{(2,0)(2.8,-0.6)(3.8,-1.2)}

\put(0.8,-3){$G\circ 2 P_2$}
\end{picture}}

\setlength{\unitlength}{10mm}
\begin{picture}(4.5,4)
\put(-1,1){\makebox(4.5,3)[b]{\usebox{\Mx}}}
\put(2.5,2){\makebox(4.5,3)[b]{\usebox{\My}}}
\put(7,2){\makebox(4.5,3)[b]{\usebox{\Mz}}}
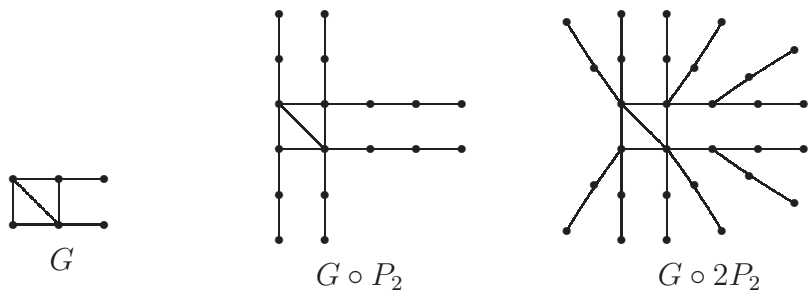
\put(4,-0.5){{\scriptsize Fig. 1:\  $2$-corona and double $2
$-corona }}
\end{picture}
\vskip 0.5cm

\vspace{3mm} We start with the following lemma.

\begin{lemma} Let $n\geq 3$ and $n_i\geq 1$ be integers for $1\leq
i\leq 2$. Then

(1) $$\gamma_{t[1,2]}(K_n)=\gamma_{t[1,2]}(K_{n_1,n_2})=2;$$

(2) $$\gamma_{t[1,2]}(P_n)=\gamma_{t[1,2]}(C_n)=\left \{
\begin{array}{ll}
\frac{n+1}{2}, & \mbox{$n\equiv 1\ (mod\ 2)$} \\
\frac{n}{2}+1, & \mbox{$n\equiv 2\ (mod\ 4)$} \\
\frac{n}{2}, & \mbox{$n\equiv 0\ (mod\ 4)$}.
\end{array}
\right.
$$

(3) if $G=H\circ 2P_2$ for a connected graph $H$ of order at least two, then
$\gamma_{t[1,2]}(G)=\frac {4n} 5$, where $n$ is the order of $G$.
\end{lemma}

\begin{proof} We leave (1) and (2) to the readers for an exercise.
To show (3), let $H$ be a connected graph of order $k\geq 2$, with
$V(H)=\{u_1, \ldots, u_k\}$. Let $V(G)=V(H)\cup X\cup Y\cup X'\cup
Y'$, where $X=\{x_1, \ldots, x_k\}$, $Y=\{y_1, \ldots, y_k\}$,
$X'=\{x_1', \ldots, x_k'\}$, $Y'=\{y_1', \ldots, y_k'\}$, and
$E(G)=E(H)\cup \{u_ix_i|\ 1\leq i\leq k\}\cup \{u_iy_i|\ 1\leq i\leq
k\}\cup \{x_ix_i'|\ 1\leq i\leq k\} \cup \{y_iy_i'|\ 1\leq i\leq
k\}$. Obviously, $X\cup Y\cup X'\cup Y'$ is a total $[1,2]$-set of $G$, so $\gamma_{t[1,2]}(G)\leq \frac{4n}{5}$. Take a $\gamma_{t[1,2]}$-set $S$ of $G$. Since $S$ is a total
dominating set of $G$, $X\cup Y\subseteq S$. On the other hand,
$V(H)\cap S=\emptyset$. If it is not true, without loss of
generality, we may assume that $u_1\in S$ and
$u_2$ is a neighbor of $u_1$ in $H$. But then, $\{x_2, y_2, u_1\}\subseteq N_G(u_2)\cap S$,
contradicting that $S$ is a total $[1,2]$-set of $G$. Combining the
above two facts with the other two facts that $X\cup Y$ is an
independent set of $G$ and $\delta(G[S])\geq 1$, it follows that
$S=X\cup Y\cup X'\cup Y'$, and thus
$\gamma_{t[1,2]}(G)=\frac{4n}{5}$.

\end{proof}

By Lemma 2.1, $\gamma_{t[1,2]}(G)\geq 2$ for any connected
graph $G$ of order $n\geq 2$, and there are infinite family of
graphs $G$ with $\gamma_{t[1,2]}(G)=2$. Indeed, for a graph $G$, $\gamma_{t[1,2]}(G)=2$
if and only if $G$ has a dominating set, which consists of a pair of adjacent vertices in $G$.
Cockayne et al. \cite{cock}
proved that $\gamma_t(G)\leq \frac{2n}{3}$ for a connected graph of
order $n\geq 3$. Later, Brigham et al. \cite{bri} characterized
those graphs achieving this bound.

\begin{theorem}(\cite{bri,cock})
Let $G$ be a connected graph of order $n\geq 3$. Then
$\gamma_t(G)\leq \frac{2n}{3}$ and the equality holds if and only if
$G$ is $C_3$, $C_6$ or $H\circ P_2$ for some connected graph $H$.
\end{theorem}

Note that $\gamma_{[1,2]}(G)\leq n$ is a trivial upper bound for a
graph $G$ of order $n$, and the equality holds for an infinite many
values of $n$, see \cite{che, wu}. In the following theorem, we
establish a sharp upper bound for the total $[1,2]$-domination
number of a connected graph in terms of its order and characterize
all graphs achieving the bound.

\begin{theorem}{\label{21}}
Let $G$ be a connected graph of order $n\geq 5$. If
$\gamma_{t[1,2]}(G)<+\infty$, then $$\gamma_{t[1,2]}(G)\leq
\frac{4n}{5},$$ with equality if and only if $G=H\circ 2P_2$ for
some connected graph $H$ of order at least two.
\end{theorem}
\begin{proof} Let $S$ be a $\gamma_{t[1,2]}(G)$-set of $G$. Since $1\leq
\delta(G[S])\leq \Delta(G[S])\leq 2$, each component of $G[S]$ is a
path or a cycle. Divide $S$ into four subsets:
$$S_1=\{u\in S|\ u\ \text{lies  in a component of} \ G[S]\ \text{isomorphic to a cycle}\},$$
$$S_2=\{u\in S|\ u\ \text{ lies in  a component of} \ G[S]\ \text{isomorphic to}\ K_2\},$$
$$S_3=\{u\in S|\ u\ \text{lies in a component of } \ G[S]\ \text{isomorphic to}\ P_3\},$$
$$S_4=\{u\in S|\ u\ \text{lies in a component of} \ G[S]\ \text{isomorphic to a path of order at least 4}\}.$$
Clearly, $|S|=\sum_{i=1}^{4}|S_i|$. Recall that for a vertex $u\in
S$, $pri(u, S)=\{v\in V(G)\setminus S\ |\  N(v)\cap S=\{u\}\}$ and
$pri(S_i, S)=\cup_{u\in S_i}pri(u, S)$.

\vspace{3mm}\noindent{\bf Claim 1.} (1) $|pri(S_1, S)|\geq |S_1|$, (2)
$|pri(S_3, S)|\geq \frac{2|S_3|}{3}$, (3) $|pri(S_4, S)|\geq
|S_4|-2\omega_4$, where $\omega_4$ denotes the number of components
in $G[S]$, which is isomorphic to a path of order at least 4.

\vspace{3mm}\noindent{\bf Proof of Claim 1.} Let $u\in S$ be a vertex. If $u\in
S_1$ and $pri(u)=\emptyset$, then $S\setminus \{u\}$ will be a total
$[1,2]$-set of $G$, contradicting the choice of $S$. Therefore
$$|pri(S_1, S)|=|\bigcup_{u\in S_1} pri(u, S)|=\sum_{u\in S_1}|pri(u, S)|\geq |S_1|.$$ This proves (1).

Now assume that $u\in S_3$ and $d_{G[S_3]}(u)=1$. If
$pri(u)=\emptyset$, $S\setminus \{u\}$ is still a total $[1,2]$-set
of $G$, contradicting the choice of $S$. Therefore
$$|pri(S_3, S)|=|\bigcup_{u\in S_3} pri(u, S)|\geq \sum_{u\in S_3,d_{G[S_3]}(u)=1}|pri(u, S)|\geq \frac {2|S_3|}{3}.$$
This proves (2).

To show (3), let $P=u_1u_2\cdots u_k$ be a component of $G[S]$,
which is isomorphic to a path of order $k$, where $k\geq 4$. By an argument
similar to the above, we have $pri(u_i, S)\neq \emptyset$ for any
$i\in\{1, 2, \ldots, k\}\setminus \{2, k-1\}$, and thus $|pri(S_4,
S)|\geq |S_4|-2\omega_4$. \hfill\qed

\vspace{3mm}\noindent{\bf Claim 2.} $|U|\geq \frac{|S_2|}{4}$ where
$U=V(G)\setminus [S\cup pri(S_1, S)\cup pri(S_3, S)\cup pri(S_4,
S)]$.

\vspace{3mm}\noindent{\bf Proof of Claim 2.} Since $G$ is connected,
every component of $G[S]$, which is isomorphic to $K_2$, has at least
one neighbor in $U$. Let $W=\cup_{u\in S_2}N(u)\cap U$. Since every
vertex of $W$ is adjacent to at most two vertices of $S$ in $G$,
$$2|W|\geq |E[S,W]|\geq |E[S_2,W]|\geq \frac {|S_2|} 2,$$
and thus $|U|\geq |W|\geq \frac{|S_2|}{4}$. \hfill\qed

Since
$\omega_4\leq \frac{|S_4|}{4}$, we have
\begin{eqnarray*}
n&=& |S|+|pri(S_1, S)|+|pri(S_3, S)|+|pri(S_4, S)|+|U|\\
&\geq & |S|+|S_1|+\frac{2|S_3|}{3}+|S_4|-2\omega_4+\frac{|S_2|}{4}\\
&\geq & \frac{5|S|}{4}+\frac{3|S_1|}{4}+\frac{5|S_3|}{12}+\frac{3|S_4|}{4}-\frac{|S_4|}{2}\\
&\geq & \frac{5|S|}{4}+\frac{3|S_1|}{4}+\frac{5|S_3|}{12}+\frac{|S_4|}{4}\\
&\geq & \frac{5|S|}{4}.
\end{eqnarray*}
So, $\gamma_{t[1,2]}(G)\leq  \frac{4n}{5}$.

If the equality holds, we have $|S_1|=|S_3|=|S_4|=0$,
$|V(G)\setminus S|=\frac{|S|}{4}$, every component isomorphic to
$K_2$ has exactly one neighbor in $V(G)\setminus S$, and each vertex
of $V(G)\setminus S$ has exactly two neighbors in $S$. And if $\frac{|S|}{4}=1$, then $G=P_5$, while $\gamma_{t[1,2]}(P_5)=3$. So, $G=H\circ
2P_2$ for a connected graph $H$ of order at least 2. By (3) of Lemma
2.1, the converse is also true.
\end{proof}

Theorem 2.3 tells us that there does not exist a graph $G$ of order
$n\geq 5$ with $\gamma_{t[1,2]}(G)=k$ for any $\lfloor\frac{4n}{5}\rfloor <k\leq n$.
However, the situation become different when $k\leq \lfloor \frac {4n} 5\rfloor$.
To see this let us construct a class of the following graphs. We start from a complete graph $K_{n-k}$ with $V(K_{n-k})=\{v_1,v_2,\cdots,v_{n-k}\}$ and denote $r=\lfloor \frac{k}{4}\rfloor$.
For $10\leq k \leq \lfloor\frac{4n}{5}\rfloor$ and $k\equiv\ 0\ (mod\  2)$, we construct a graph $F_{n,k}$ of order
$n$ with $\gamma_{t[1,2]}(F_{n,k})=k$ as follows:
$F_{n,k}$ has $$V(F_{n,k})=V(H)\cup \{w_1, \ldots, w_{\frac k
2}\}\cup\{w_1', \ldots, w_{\frac k 2}'\}$$ and
$E(F_{n,k})=E(H)\cup\{w_iw_i'|\ 1\leq i\leq \frac{k}{2}\}\cup E'$,
where
$$E'=\{v_iw_{2i-1}, v_iw_{2i}|\ 1\leq i\leq r-1\}\cup \{v_rw_{2r-1}\}\cup
\{v_jw_{2r}, v_jw_{2r}'|\ r+1\leq j\leq n-k \} $$ if $k\equiv 0\ (mod\ 4)$; and
$$E'=\{v_iw_{2i-1},v_iw_{2i}|\ 1\leq i\leq r\}\cup \{v_jw_{2r+1}
v_jw_{2r+1}'|\ r+1\leq j\leq n-k\}$$ if $k\equiv 2\ (mod \ 4)$.
For $10\leq k < \lfloor\frac{4n}{5}\rfloor$ and $k\equiv\ 1\ (mod\  2)$, we construct a graph $F_{n,k}$ of order
$n$ with $\gamma_{t[1,2]}(F_{n,k})=k$ as follows. Note that now $n-k-r\geq 2$. $F_{n,k}$ has $$V(F_{n,k})=V(H)\cup
\{w_1, \ldots, w_{\frac {k-3} 2}\}\cup\{w_1', \ldots, w_{\frac {k-3}
2}'\}\cup \{w, w', w''\}$$ and
$E(F_{n,k})=E(H)\cup\{w_iw_i'|\ 1\leq
i\leq \frac{k-3}{2}\}\cup\{ww', w'w''\}\cup E'$,
where
\begin{eqnarray*}
E'&=&\{v_iw_{2i-1},v_iw_{2i}|\ 1\leq i\leq r-1\}\cup \{v_rw_{2r-1},v_{r+1}w, v_{r+2}w''\}\\
&& \cup \{v_jw'|\ r+3\leq j\leq n-k \}
\end{eqnarray*} if $k\equiv\ 1\ (mod \ 4)$; and
\begin{eqnarray*}
E'&=& \{v_iw_{2i-1},v_iw_{2i}|\ 1\leq i\leq r\}\cup \{v_{r+1}w, v_{r+2}w''\}\\
&& \cup \{v_jw'|\
r+3\leq j\leq n-k \}
\end{eqnarray*}
if $k\equiv\ 3\ (mod\ 4)$.

In Fig. 2 we display $F_{14,10}$, $F_{15,11}$, $F_{16,12}$ and
$F_{18,13}$.

\setlength{\unitlength}{6mm}\newsavebox{\Mc} \savebox{\Mc}
{\begin{picture}(0,0) \put(0,0){\circle*{0.2}}
\put(2,0){\circle*{0.2}}\put(2,2){\circle*{0.2}}
\put(0,2){\circle*{0.2}}
\put(0,4){\circle*{0.2}}\put(2,4){\circle*{0.2}}
\put(-0.5,-1){\circle*{0.2}}\put(0.5,-1){\circle*{0.2}}
\put(1.5,-1){\circle*{0.2}}\put(2.5,-1){\circle*{0.2}}
\put(-0.75,-2){\circle*{0.2}}\put(0.75,-2){\circle*{0.2}}
\put(1.25,-2){\circle*{0.2}}\put(2.75,-2){\circle*{0.2}}

\qbezier{(0,0)(1,0)(2,0)} \qbezier{(0,0)(0,0.5)(0,4)}
\qbezier{(0,0)(1,1)(2,2)} \qbezier{(0,2)(1,2)(2,2)}
\qbezier{(0,2)(1,1)(2,0)} \qbezier{(2,0)(2,1)(2,4)}
\qbezier{(0,2)(1,3)(2,4)} \qbezier{(0,4)(1,3)(2,2)}
\qbezier{(0,4)(1,4)(2,4)}\qbezier{(0,0)(-0.25,-0.5)(-0.5,-1)}\qbezier{(0,0)(0.25,-0.5)(0.5,-1)}
\qbezier{(2,0)(1.75,-0.5)(1.5,-1)}
\qbezier{(2,0)(2.25,-0.5)(2.5,-1)}\qbezier{(-0.5,-1)(-0.62,-1.5)(-0.75,-2)}
\qbezier{(0.5,-1)(0.62,-1.5)(0.75,-2)}
\qbezier{(1.5,-1)(1.32,-1.5)(1.25,-2)}\qbezier{(2.5,-1)(2.62,-1.5)(2.75,-2)}

\put(-0.7,0){$v_1$}\put(2.1,0){$v_2$}\put(2.1,2){$v_3$}\put(-0.7,2){$v_4$}
\put(-1.3,-1){$w_1$}\put(0.5,-0.9){$w_2$}\put(1.6,-1.2){$w_3$}\put(2.6,-1){$w_4$}
\put(-1.5,-2){$w_1'$}\put(-0.1,-2){$w_2'$}\put(1.4,-2){$w_3'$}\put(2.8,-2){$w_4'$}
\put(-0.8,4){$w_5$}\put(2.1,4){$w_5'$}
\put(0.4,-2.9){ $F_{14,10}$}

\end{picture}}

\setlength{\unitlength}{6mm}\newsavebox{\Md} \savebox{\Md}
{\begin{picture}(0,0) \put(0,0){\circle*{0.2}}
\put(2,0){\circle*{0.2}}\put(2,2){\circle*{0.2}}
\put(0,2){\circle*{0.2}}
\put(0,4){\circle*{0.2}}\put(2,4){\circle*{0.2}}\put(1,4){\circle*{0.2}}
\put(-0.5,-1){\circle*{0.2}}\put(0.5,-1){\circle*{0.2}}
\put(1.5,-1){\circle*{0.2}}\put(2.5,-1){\circle*{0.2}}
\put(-0.75,-2){\circle*{0.2}}\put(0.75,-2){\circle*{0.2}}
\put(1.25,-2){\circle*{0.2}}\put(2.75,-2){\circle*{0.2}}

\qbezier{(0,0)(1,0)(2,0)} \qbezier{(0,0)(0,0.5)(0,4)}
\qbezier{(0,0)(1,1)(2,2)} \qbezier{(0,2)(1,2)(2,2)}
\qbezier{(0,2)(1,1)(2,0)} \qbezier{(2,0)(2,1)(2,4)}
\qbezier{(0,4)(1,4)(2,4)}\qbezier{(0,0)(-0.25,-0.5)(-0.5,-1)}\qbezier{(0,0)(0.25,-0.5)(0.5,-1)}
\qbezier{(2,0)(1.75,-0.5)(1.5,-1)}
\qbezier{(2,0)(2.25,-0.5)(2.5,-1)}\qbezier{(-0.5,-1)(-0.62,-1.5)(-0.75,-2)}
\qbezier{(0.5,-1)(0.62,-1.5)(0.75,-2)}
\qbezier{(1.5,-1)(1.32,-1.5)(1.25,-2)}\qbezier{(2.5,-1)(2.62,-1.5)(2.75,-2)}

\put(-0.7,0){$v_1$}\put(2.1,0){$v_2$}\put(2.1,2){$v_3$}\put(-0.7,2){$v_4$}
\put(-1.3,-1){$w_1$}\put(0.5,-0.9){$w_2$}\put(1.6,-1.2){$w_3$}\put(2.6,-1){$w_4$}
\put(-1.5,-2){$w_1'$}\put(-0.1,-2){$w_2'$}\put(1.4,-2){$w_3'$}\put(2.8,-2){$w_4'$}
\put(-0.8,4){$w''$}\put(2.1,4){$w$}\put(0.8,4.2){$w'$}
\put(0.4,-2.9){$F_{15,11}$}
\end{picture}}

\setlength{\unitlength}{6mm}\newsavebox{\Mj} \savebox{\Mj}
{\begin{picture}(0,0) \put(0,0){\circle*{0.2}}
\put(2,0){\circle*{0.2}}\put(2,2){\circle*{0.2}}
\put(0,2){\circle*{0.2}}
\put(-1,4){\circle*{0.2}}\put(2,4){\circle*{0.2}}
\put(-0.5,-1){\circle*{0.2}}\put(0.5,-1){\circle*{0.2}}
\put(1.5,-1){\circle*{0.2}}\put(2.5,-1){\circle*{0.2}}
\put(-0.75,-2){\circle*{0.2}}\put(0.75,-2){\circle*{0.2}}
\put(1.25,-2){\circle*{0.2}}\put(2.75,-2){\circle*{0.2}}

\qbezier{(0,0)(1,0)(2,0)} \qbezier{(0,0)(0,0.5)(0,2)}\qbezier{(0,2)(-0.5,3)(-1,4)}
\qbezier{(0,0)(1,1)(2,2)} \qbezier{(0,2)(1,2)(2,2)}
\qbezier{(0,2)(1,1)(2,0)} \qbezier{(2,0)(2,1)(2,4)}
\qbezier{(0,2)(0.5,3)(1,4)} \qbezier{(-1,4)(0.4,4)(1,4)}\qbezier{(0,0)(-0.25,-0.5)(-0.5,-1)}\qbezier{(0,0)(0.25,-0.5)(0.5,-1)}
\qbezier{(2,0)(1.75,-0.5)(1.5,-1)}
\qbezier{(2,0)(2.25,-0.5)(2.5,-1)}\qbezier{(-0.5,-1)(-0.62,-1.5)(-0.75,-2)}
\qbezier{(0.5,-1)(0.62,-1.5)(0.75,-2)}
\qbezier{(1.5,-1)(1.32,-1.5)(1.25,-2)}\qbezier{(2.5,-1)(2.62,-1.5)(2.75,-2)}

\put(-0.7,0){$v_1$}\put(2.1,0){$v_2$}\put(2.1,2){$v_3$}\put(-0.7,2){$v_4$}
\put(-1.3,-1){$w_1$}\put(0.5,-0.9){$w_2$}\put(1.6,-1.2){$w_3$}\put(2.6,-1){$w_4$}
\put(-1.5,-2){$w_1'$}\put(-0.1,-2){$w_2'$}\put(1.4,-2){$w_3'$}\put(2.8,-2){$w_4'$}
\put(-1.8,4){$w_6$}\put(2.1,4){$w_5'$} \put(2,3){\circle*{0.2}}\put(2.2,3){$w_5$} \put(1,4){\circle*{0.2}} \put(0.6,4.2){$w_6'$}
\put(0.4,-2.9){$F_{16,12}$}
\end{picture}}

\setlength{\unitlength}{6mm}\newsavebox{\Mk} \savebox{\Mk}
{\begin{picture}(0,0) \put(0,0){\circle*{0.2}}
\put(2,0){\circle*{0.2}}\put(2,2){\circle*{0.2}}
\put(0,2){\circle*{0.2}}
\put(-1,4){\circle*{0.2}}\put(2,4){\circle*{0.2}}
\put(-0.5,-1){\circle*{0.2}}\put(0.5,-1){\circle*{0.2}}
\put(1.5,-1){\circle*{0.2}}\put(2.5,-1){\circle*{0.2}}
\put(-0.75,-2){\circle*{0.2}}\put(0.75,-2){\circle*{0.2}}
\put(1.25,-2){\circle*{0.2}}\put(2.75,-2){\circle*{0.2}}

\qbezier{(0,0)(1,0)(2,0)} \qbezier{(0,0)(0,0.5)(0,2)}\qbezier{(-1,1)(-1,3)(-1,4)}
\qbezier{(0,0)(1,1)(2,2)} \qbezier{(0,2)(1,2)(2,2)}
\qbezier{(0,2)(1,1)(2,0)} \qbezier{(2,0)(2,1)(2,4)}
\qbezier{(0,2)(0.5,3)(1,4)} \qbezier{(-1,4)(0.4,4)(1,4)}\qbezier{(0,0)(-0.25,-0.5)(-0.5,-1)}\qbezier{(0,0)(0.25,-0.5)(0.5,-1)}
\qbezier{(2,0)(1.75,-0.5)(1.5,-1)}
\qbezier{(2,0)(2.25,-0.5)(2.5,-1)}\qbezier{(-0.5,-1)(-0.62,-1.5)(-0.75,-2)}
\qbezier{(0.5,-1)(0.62,-1.5)(0.75,-2)}
\qbezier{(1.5,-1)(1.32,-1.5)(1.25,-2)}\qbezier{(2.5,-1)(2.62,-1.5)(2.75,-2)}
\qbezier{(-1,1)(-0.5,0.5)(0,0)}\qbezier{(-1,1)(-0.5,1.5)(0,2)} \qbezier{(-1,1)(0.5,0.5)(2,0)}\qbezier{(-1,1)(0.5,1.5)(2,2)}

\put(-0.7,-0.1){$v_1$}\put(2.1,0){$v_2$}\put(2.1,2){$v_3$}\put(-0.7,2){$v_4$}
\put(-1.3,-1){$w_1$}\put(0.5,-0.9){$w_2$}\put(1.6,-1.2){$w_3$}\put(2.6,-1){$w_4$}
\put(-1.5,-2){$w_1'$}\put(-0.1,-2){$w_2'$}\put(1.4,-2){$w_3'$}\put(2.8,-2){$w_4'$}
\put(-1.8,4){$w''$}\put(2.1,4){$w_5'$} \put(2,3){\circle*{0.2}}\put(2.2,3){$w_5$} \put(0,4){\circle*{0.2}}\put(-0.4,4.2){$w'$} \put(1,4){\circle*{0.2}} \put(0.6,4.2){$w$}\put(-1,1){\circle*{0.2}}\put(-1.7,1){$v_5$}
\put(0.4,-2.9){$F_{18,13}$}
\end{picture}}

\setlength{\unitlength}{10mm}
\begin{picture}(4.5,5.4)
\put(0.6,2){\makebox(4.5,3)[b]{\usebox{\Mc}}}
\put(6.2,2){\makebox(4.5,3)[b]{\usebox{\Md}}}
\end{picture}
\vspace{3mm}
\setlength{\unitlength}{10mm}

\begin{picture}(4.5,5.4)
\put(0.6,2){\makebox(4.5,3)[b]{\usebox{\Mj}}}
\put(6.2,2){\makebox(4.5,3)[b]{\usebox{\Mk}}}
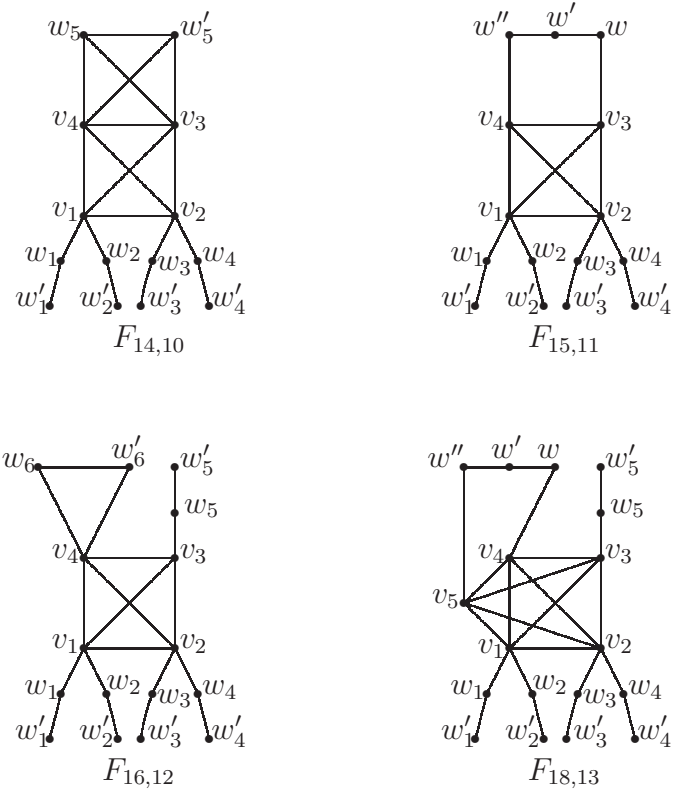
\put(3,-0.3){{\scriptsize Fig. 2. The graphs $F_{14,10}$,
$F_{15,11}$, $F_{16,12}$ and $F_{18,13}$. }}
\end{picture}
\vspace{3mm}

\begin{proposition} For any two positive integers $n$ and $k$ with $10\leq k\leq
\frac{4n}{5}$, there exists a graph $G$ of order $n$ with
$\gamma_{t[1,2]}(G)=k$.
\end{proposition}

\begin{proof} Note that the order of the graph $F_{n,k}$
as constructed above is $n$ and $$W=\left\{
\begin{array}{ll}
\{w_1, \ldots, w_{\frac k 2}\}\cup\{w_1',
\ldots, w_{\frac k 2}'\}, & \mbox{$k\equiv 0\quad (mod\ 2)$}\\
\{w_1, \ldots, w_{\frac {k-3} 2}\}\cup\{w_1', \ldots, w_{\frac {k-3}
2}'\}\cup \{w, w', w''\}, & \mbox{$k\equiv\ 1\quad (mod\ 2)$},\\
\end{array}
\right.
$$ is a total $[1,2]$-set of $G$. So, $\gamma_{t[1,2]}(F_{n,k})\leq k$.
Take a $\gamma_{t[1,2]}$-set $S$ of $F_{n,k}$.

\vspace{3mm}\noindent{\bf Case 1.} $k\equiv\ 0\ (mod\ 2)$.

\vspace{3mm} Since $S$ is a total dominating set, $\{w_1,w_2,\ldots,
w_{\frac{k-2}{2}}\}\subseteq S$. On the other hand, $V(K_{n-k})\cap
S=\emptyset$. Otherwise, if $v_i\in S$ for some $1\leq i\leq n-k$,
then $\{v_i,w_1,w_2\}\subseteq N(v_1)\cap S$ if $i\neq 1$ or
$\{v_i,w_3,w_4\}\subseteq N(v_2)\cap S$ if $i=1$, both contradicting
that $S$ is a total $[1,2]$-set of $G$. Then Combining the above two
facts with the other fact that $\{w_1,w_2,\ldots,w_{\frac{k}{2}}\}$
is an independent set of $F_{n,k}$, it follows that
$S=\{w_1,w_2,\ldots,w_{\frac{k}{2}}\}\cup
\{w_1',w_2',\ldots,w_{\frac{k}{2}}'\}$, and thus
$\gamma_{t[1,2]}(F_{n,k})=k$.

\vspace{3mm}\noindent{\bf Case 2.} $k\equiv\ 1\ (mod\ 2)$.

\vspace{3mm} Since $S$ is a total dominating set, $\{w_1,w_2,\ldots,
w_{\frac{k-3}{2}}\}\subseteq S$. On the other hand, $V(K_{n-k})\cap
S=\emptyset$. Otherwise, if $v_i\in S$ for some $1\leq i\leq n-k$,
then $\{v_i,w_1,w_2\}\subseteq N(v_1)\cap S$ if $i\neq 1$ or
$\{v_i,w_3,w_4\}\subseteq N(v_2)\cap S$ if $i=1$, both contradicting
that $S$ is a total $[1,2]$-set. Then Combining the above two facts
with the other fact that
$\{w_1,w_2,\ldots,w_{\frac{k-3}{2}},w,w''\}$ is an independent set
of $F_{n,k}$, it follows that
$S=\{w_1,w_2,\ldots,w_{\frac{k-3}{2}}\}\cup
\{w_1',w_2',\ldots,w_{\frac{k-3}{2}}'\}\cup \{w,w',w''\}$, and thus
$\gamma_{t[1,2]}(F_{n,k})=k$.
\end{proof}

The upper bound on $\gamma_{t[1,2]}(G)$ in Theorem \ref{21} is sharp. However, every
extremal graph $G$ with $\gamma_{t[1,2]}(G)=\frac{4n}{5}$ satisfies that $\delta(G)=1$. It is natural to consider the
upper bound on $\gamma_{t[1,2]}(G)$ when $\delta(G)\geq 2$.

\begin{theorem}{\label{22}}
Let $G$ be a connected graph of order $n\geq 3$ and $\delta(G)\geq
2$. If $\gamma_{t[1,2]}(G)<+\infty$, then $$\gamma_{t[1,2]}(G)\leq
\frac{2n}{3}.$$
\end{theorem}

\begin{proof} Let $S$ be a $\gamma_{t[1,2]}(G)$-set of $G$. Same as the proof of
Theorem \ref{21}, we divide $S$ into four subsets $S_1$,
$S_2$, $S_3$ and $S_4$.

\vspace{3mm}\noindent{\bf Claim 1.} (1) $|pri(S_1, S)|\geq |S_1|$,
(2) $|pri(S_3, S)|\geq \frac{2|S_3|}{3}$, (3) $|pri(S_4, S)|\geq
|S_4|-2\omega_4$, where $\omega_4$ denotes the number of components
in $G[S]$, which is isomorphic to a path of order at least four.

\vspace{3mm}\noindent{\bf Claim 2.} $|U|\geq \frac{|S_2|}{2}$ where
$U=V(G)\setminus (S\cup pri(S_1, S)\cup pri(S_3, S)\cup pri(S_4,
S))$.

\vspace{3mm}\noindent{\bf Proof of Claim 2.} Since $\delta(G)\geq
2$, every vertex in $S_2$ has at least one neighbor in $U$. Let
$W=\cup_{u\in S_2}N(u)\cap U$. Since every vertex in $W$ is adjacent
to at most $2$ vertices in $S$ and $\delta(G)\geq 2$, we have
$$2|W|\geq |E[S,W]|\geq |E[S_2,W]|\geq \sum_{v\in S_2} (d(v)-1)\geq
\sum_{v\in S_2} 1\geq |S_2|,$$ and thus $|U|\geq |W|\geq
\frac{|S_2|}{2}$. \hfill\qed

Since $\omega_4\leq \frac{|S_4|}{4}$, we have
\begin{eqnarray*}
n&=& |S|+|pri(S_1, S)|+|pri(S_3, S)|+|pri(S_4, S)|+|U|\\
&\geq & |S|+|S_1|+\frac{2|S_3|}{3}+|S_4|-2\omega_4+\frac{|S_2|}{2}\\
&\geq & \frac{3|S|}{2}+\frac{|S_1|}{2}+\frac{|S_3|}{6}\\
&\geq & \frac{3|S|}{2}.
\end{eqnarray*}

So $\gamma_{t[1,2]}(G)\leq \frac{2n}{3}$ and the
equality holds only if $S_1=S_3=\emptyset$, $S_4=4\omega_4$,
$|pri(S_4)|=\frac{|S_4|}{2}$ and $|U|=|W|=\frac{|S_2|}{2}$.

\end{proof}

This upper bound is also sharp. For an integer $k\geq 4$, we can
construct a graph $F_k$ of order $n=3k$ with
$\gamma_{t[1,2]}(F_k)=2k$ as follows. We start from the complete
graph $K_k$ with $V(K_k)=\{v_1,v_2,\cdots,v_k\}$. Let
$V(F_k)=V(K_k)\cup W\cup W'$, where $W=\{w_1, w_2, \cdots, w_k\}$
and $W'= \{w_1', w_2', \cdots, w_{k}'\}$, and $E(F_k)=
\cup_{i=1}^k\{v_iw_i, w_iw_i', w_i'v_{i+1}\}\cup E(K_k),$ where
$i+1$ is taken modulo $k$. Therefore, $F_k$ is a graph of order $3k$
and for any $1\leq i\leq k$, $d(v_i)=k+1$ and $d(w_i)=d(w_i')=2$.
The graph $F_4$ is shown in Fig. 3.

\setlength{\unitlength}{6mm}\newsavebox{\Ma} \savebox{\Ma}
{\begin{picture}(0,0) \put(0,0){\circle*{0.2}}
\put(2,0){\circle*{0.2}}\put(2,2){\circle*{0.2}}
\put(0,2){\circle*{0.2}}
\put(0.5,3){\circle*{0.2}}\put(1.5,3){\circle*{0.2}}
\put(-1,1.5){\circle*{0.2}}\put(-1,0.5){\circle*{0.2}}
\put(0.5,-1){\circle*{0.2}}\put(1.5,-1){\circle*{0.2}}
\put(3,0.5){\circle*{0.2}}\put(3,1.5){\circle*{0.2}}

\qbezier{(0,0)(1,0)(2,0)} \qbezier{(0,0)(0,0.5)(0,2)}
\qbezier{(0,0)(1,1)(2,2)} \qbezier{(0,2)(1,2)(2,2)}
\qbezier{(0,2)(1,1)(2,0)} \qbezier{(2,0)(2,1)(2,2)}
\qbezier{(0,2)(-0.5,1.75)(-1,1.5)}
\qbezier{(-1,1.5)(-1,1)(-1,0.5)}\qbezier{(0,0)(-0.5,0.25)(-1,0.5)}
\qbezier{(2,0)(1.75,-0.5)(1.5,-1)}
\qbezier{(0.5,-1)(1,-1)(1.5,-1)}\qbezier{(0,0)(0.25,-0.5)(0.5,-1)}
\qbezier{(2,0)(2.5,0.25)(3,0.5)}
\qbezier{(3,0.5)(3,1)(3,1.5)}\qbezier{(2,2)(2.5,1.75)(3,1.5)}
\qbezier{(2,2)(1.75,2.5)(1.5,3)}
\qbezier{(0.5,3)(1,3)(1.5,3)}\qbezier{(0,2)(0.25,2.5)(0.5,3)}

\put(-0.7,-0.4){$v_1$}\put(2.1,-0.4){$v_2$}\put(2.1,2){$v_3$}\put(-0.8,2){$v_4$}
\put(-1.9,0.4){$w_4'$}\put(-1.9,1.4){$w_4$}\put(0.1,-1.4){$w_1$}\put(1.4,-1.4){$w_1'$}
\put(1.4,3.2){$w_3$}\put(0,3.2){$w_3'$}\put(3.1,0.4){$w_2$}\put(3.1,1.4){$w_2'$}
\end{picture}}

\setlength{\unitlength}{6mm}\newsavebox{\Mb} \savebox{\Mb}
{\begin{picture}(0,0) \put(0,0){\circle*{0.2}}
\put(2,0){\circle*{0.2}}\put(2,2){\circle*{0.2}}
\put(0,2){\circle*{0.2}}
\put(0.5,3){\circle*{0.2}}\put(1.5,3){\circle*{0.2}}
\put(-1,1.5){\circle*{0.2}}\put(-1,0.5){\circle*{0.2}}
\put(0.5,-1){\circle*{0.2}}\put(1.5,-1){\circle*{0.2}}
\put(3,0.5){\circle*{0.2}}\put(3,1.5){\circle*{0.2}}

\put(0,4){\circle*{0.2}}\put(2,4){\circle*{0.2}}
\qbezier{(0,2)(0,3)(0,4)} \qbezier{(2,2)(2,3)(2,4)}
\qbezier{(0,4)(1,4)(2,4)}
\put(-2,2){\circle*{0.2}}\put(-2,0){\circle*{0.2}}
\qbezier{(0,2)(-1,2)(-2,2)} \qbezier{(0,0)(-1,0)(-2,0)}
\qbezier{(-2,0)(-2,1)(-2,2)}
\put(0,-2){\circle*{0.2}}\put(2,-2){\circle*{0.2}}
\qbezier{(0,0)(0,-1)(0,-2)} \qbezier{(2,0)(2,-1)(2,-2)}
\qbezier{(0,-2)(1,-2)(2,-2)}
\put(4,2){\circle*{0.2}}\put(4,0){\circle*{0.2}}
\qbezier{(2,0)(3,0)(4,0)} \qbezier{(2,2)(3,2)(4,2)}
\qbezier{(4,0)(4,1)(4,2)}

\qbezier{(0,0)(1,0)(2,0)} \qbezier{(0,0)(0,0.5)(0,2)}
\qbezier{(0,0)(1,1)(2,2)} \qbezier{(0,2)(1,2)(2,2)}
\qbezier{(0,2)(1,1)(2,0)} \qbezier{(2,0)(2,1)(2,2)}
\qbezier{(0,2)(-0.5,1.75)(-1,1.5)}
\qbezier{(-1,1.5)(-1,1)(-1,0.5)}\qbezier{(0,0)(-0.5,0.25)(-1,0.5)}
\qbezier{(2,0)(1.75,-0.5)(1.5,-1)}
\qbezier{(0.5,-1)(1,-1)(1.5,-1)}\qbezier{(0,0)(0.25,-0.5)(0.5,-1)}
\qbezier{(2,0)(2.5,0.25)(3,0.5)}
\qbezier{(3,0.5)(3,1)(3,1.5)}\qbezier{(2,2)(2.5,1.75)(3,1.5)}
\qbezier{(2,2)(1.75,2.5)(1.5,3)}
\qbezier{(0.5,3)(1,3)(1.5,3)}\qbezier{(0,2)(0.25,2.5)(0.5,3)}
\put(0.8,-3){$F_4'$}
\end{picture}}
\setlength{\unitlength}{10mm}
\begin{picture}(4.5,3.5)
\put(3.3,1){\makebox(4.5,3)[b]{\usebox{\Ma}}}
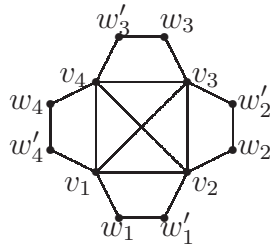
\put(4.9,-0.5){{\scriptsize Fig. 3. The graph $F_4$ }}
\end{picture}
\vskip 0.5cm

\vspace{4mm}

\begin{proposition} \label{22} For an integer $k\geq 4$,
$\gamma_{t[1,2]}(F_k)=\frac{2n}{3}$, where $n=3k$.
\end{proposition}
\begin{proof} Let $S=W\cup W'$. One can see that $S$ is a total $[1,2]$-set of
$F_k$ and thus $\gamma_{t[1,2]}(F_k)\leq \frac{2n}{3}$. Now suppose
that $\gamma_{t[1,2]}(F_k)<\frac{2n}{3}$ and $S'$ is a
$\gamma_{t[1,2]}$-set of $F_k$. Then $S'\cap V(K_k)\neq \emptyset$
and $|S'\cap V(K_k)|\leq 2$.

Assume first that $|V(K_k)\cap S'|=2$ and let $v_i, v_j\in S'$ for
some two distinct $i, j\in \{1, 2, \cdots, k\}$. In this case,
$(\{v_1, v_2,\cdots, v_k\}\setminus \{v_i, v_j\})\cap S'=\emptyset$
and thus $w_l, w_l'\in S'$ for any $l\in \{1, 2, \cdots,
k\}\setminus \{i, j\}$. But, $|N_G(v_l)\cap S'|\geq 3$, a
contradiction. Now assume that $|V(K_k)\cap S'|=1$ and $V(K_k)\cap
S'=\{v_1\}$. Since $v_i\not\in S'$ for any $i\in\{2, 3, 4\}$, we
have $w_2, w_2'\in S'$ and $w_3, w_3'\in S'$, and thus
$|N_G(v_3)\cap S'|\geq 3$, a contradiction.

\end{proof}

Now, for $8\leq k\leq \frac{2n}{3}-1$, we construct a graph
$H_{n,k}$ of order $n$, $\gamma_{t[1,2]}(H_{n,k})=k$ and
$\delta(H_{n,k})\geq 2$ as follows. We start from the complete graph
$K_{n-k}$ with $V(K_{n-k})=\{v_1,v_2,\cdots,v_{n-k}\}$ and
$r=\lfloor\frac{k-2}{2}\rfloor$. The graph $H_{n,k}$ has
$V(H_{n,k})=V(K_{n_k})\cup \{w_1, \ldots, w_r\}\cup \{w_1', \ldots,
w_r'\}\cup V'$, where
$$V'=\left\{
\begin{array}{ll}
\{w, w'\}, & \mbox{if\ $k=2r+2$ }\\
\{w, w', w''\}, & \mbox{ if\ $k=2r+3$ },\\
\end{array}
\right.
$$ and $$E(H_{n,k})=E(K_{n-k})\cup\{v_iw_i, w_iw_i',
w_i'v_{i+1}|\ 1\leq i\leq r-1\}\cup \{v_rw_r,w_rw_r',w_r'v_1\}\cup  E',$$ where
$$E'=
\{ww'\}\cup \{v_jw,v_jw'|\ r+1\leq j\leq
n-k\},$$ if $ k=2r+2$; and $$E'=
\{v_{r+1}w, v_{r+2}w'',ww',w'w''\}
\cup \{v_jw,v_jw''|\ r+3\leq j\leq n-k \}, $$ if $k=2r+3$.
See Fig. 4 for an
illustration for $H_{14,8}$ and $H_{15,9}$.

\setlength{\unitlength}{6mm}\newsavebox{\Me} \savebox{\Me}
{\begin{picture}(0,0) \put(0,0){\circle*{0.2}}
\put(2,0){\circle*{0.2}}\put(2,2){\circle*{0.2}}
\put(0,2){\circle*{0.2}}
\put(0,4){\circle*{0.2}}\put(2,4){\circle*{0.2}}
\put(-1,3.5){\circle*{0.2}}\put(-1,2.5){\circle*{0.2}}
\put(-1,1.5){\circle*{0.2}}\put(-1,0.5){\circle*{0.2}}
\put(-2,3){\circle*{0.2}}\put(-2,1){\circle*{0.2}}
\put(4,3){\circle*{0.2}}\put(4,1){\circle*{0.2}}
\put(0.7,1.5){$K_6$}

\qbezier{(0,0)(1,0)(2,0)} \qbezier{(0,0)(0,0.5)(0,4)}
\qbezier{(2,0)(2,1)(2,4)} \qbezier{(0,4)(1,4)(2,4)}
\qbezier{(0,0)(-0.5,0.25)(-1,0.5)}\qbezier{(0,0)(-1.2,-0.2)(-2,1)}
\qbezier{(0,2)(-0.5,1.75)(-1,1.5)}
\qbezier{(0,2)(-0.5,2.25)(-1,2.5)}\qbezier{(-1,0.5)(-1,1)(-1,1.5)}
\qbezier{(0,4)(-0.5,3.75)(-1,3.5)}
\qbezier{(0,4)(-1.2,4.2)(-2,3)}\qbezier{(-1,3.5)(-1,3)(-1,2.5)}
\qbezier{(-2,3)(-2,2)(-2,1)} \qbezier{(2,0)(3,1.5)(4,3)}\qbezier{(2,2)(3,1.5)(4,1)}
\qbezier{(4,3)(3,2.5)(2,2)}
\qbezier{(4,3)(3,3.5)(2,4)}\qbezier{(4,1)(3,0.5)(2,0)}
\qbezier{(4,1)(3,2.5)(2,4)}\qbezier{(4,1)(4,2)(4,3)}
\put(0.4,-1.1){$H_{14,8}$}\put(0,4.1){$v_1$} \put(0,2.2){$v_2$}\put(0,-0.4){$v_3$} \put(2,-0.4){$v_4$}\put(1.4,2.2){$v_5$} \put(2,4.1){$v_6$}
\put(4.1,1){$w$} \put(4.1,3){$w'$}

\end{picture}}

\setlength{\unitlength}{6mm}\newsavebox{\Mf} \savebox{\Mf}
{\begin{picture}(0,0) \put(0,0){\circle*{0.2}}
\put(2,0){\circle*{0.2}}\put(2,2){\circle*{0.2}}
\put(0,2){\circle*{0.2}}
\put(0,4){\circle*{0.2}}\put(2,4){\circle*{0.2}}
\put(-1,3.5){\circle*{0.2}}\put(-1,2.5){\circle*{0.2}}
\put(-1,1.5){\circle*{0.2}}\put(-1,0.5){\circle*{0.2}}
\put(-2,3){\circle*{0.2}}\put(-2,1){\circle*{0.2}}
\put(4,3){\circle*{0.2}}\put(4,1){\circle*{0.2}}
\put(4,2){\circle*{0.2}} \put(0.7,1.5){$K_6$}

\qbezier{(0,0)(1,0)(2,0)} \qbezier{(0,0)(0,0.5)(0,4)}
\qbezier{(2,0)(2,1)(2,4)} \qbezier{(0,4)(1,4)(2,4)}
\qbezier{(0,0)(-0.5,0.25)(-1,0.5)}\qbezier{(0,0)(-1.2,-0.2)(-2,1)}
\qbezier{(0,2)(-0.5,1.75)(-1,1.5)}
\qbezier{(0,2)(-0.5,2.25)(-1,2.5)}\qbezier{(-1,0.5)(-1,1)(-1,1.5)}
\qbezier{(0,4)(-0.5,3.75)(-1,3.5)}
\qbezier{(0,4)(-1.2,4.2)(-2,3)}\qbezier{(-1,3.5)(-1,3)(-1,2.5)}
\qbezier{(-2,3)(-2,2)(-2,1)}
\qbezier{(4,3)(3,2.5)(2,2)}
\qbezier{(4,1)(3,0.5)(2,0)}\qbezier{(2,4)(3,2.5)(4,1)}\qbezier{(2,4)(3,3.5)(4,3)}
\qbezier{(4,1)(4,2)(4,3)}
\put(0.4,-1){$H_{15,9}$}\put(0,4.1){$v_1$} \put(0,2.2){$v_2$}\put(0,-0.4){$v_3$} \put(2,-0.4){$v_4$}\put(1.4,2.2){$v_5$} \put(2,4.1){$v_6$}
\put(4.1,1){$w$} \put(4.1,3){$w''$}\put(4.1,2){$w'$}

\end{picture}}

\setlength{\unitlength}{10mm}
\begin{picture}(4.5,5)
\put(0.5,2){\makebox(4.5,3)[b]{\usebox{\Me}}}
\put(6.5,2){\makebox(4.5,3)[b]{\usebox{\Mf}}}
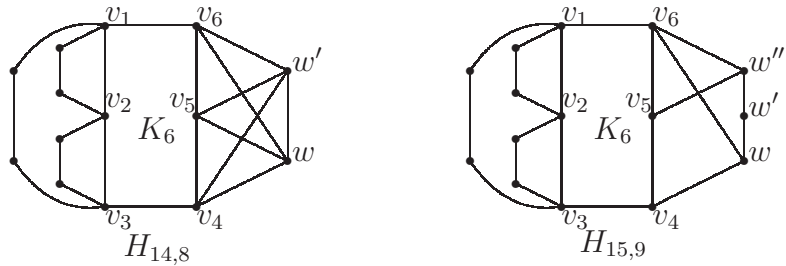
\put(4.5,0.7){{\scriptsize Fig. 4. The graphs $H_{14,8}$ and
$H_{15,9}$}}
\end{picture}

\begin{proposition} For any $n$ and $k$ with $8\leq k\leq \frac{2n}{3}-1$,
there exists a graph $G$ of order $n$
with $\delta(G)\geq 2$ and $\gamma_{t[1,2]}(G)=k$.
\end{proposition}
\begin{proof} Note that the order of the graph $H_{n,k}$ as
constructed above is
$n$ and $$S'=\left\{
\begin{array}{ll}
\{w_1, \ldots, w_r\}\cup\{w_1',
\ldots, w_r'\}\cup \{w, w'\}, & \mbox{if $k=2r+2$}\\
\{w_1, \ldots, w_r\}\cup\{w_1', \ldots, w_r'\}\cup \{w, w', w''\}, &
\mbox{if $k=2r+3$},\\
\end{array}
\right.
$$ is a total $[1,2]$-set of $H_{n,k}$. So, $\gamma_{t[1,2]}(H_{n,k})\leq k$.
In addition, any proper subset of $S'$ is not a total $[1,2]$-set of $H_{n,k}$.
Let $S$ be a $\gamma_{t[1,2]}$-set of $H_{n,k}$.
Note first that $|S\cap V(K_{n-k})|\leq 2$ since $S$ is a $[1,2]$-set of $H_{n,k}$.
Our proof is based on the fact that: if $v_i\notin S$ and $v_{i+1}\notin S$
for $1\leq i\leq r-1$, then $\{w_i, w_{i}'\}\subseteq  S$ since $N(w_i)=\{v_i,w_i'\}$, $N(w_i')=\{w_i,v_{i+1}\}$ and
$S$ is a total $[1,2]$-set; and similarly if $S\cap \{v_{r+1},v_{r+2},\ldots,v_{n-k}\}=\emptyset$,
then $\{w,w'\}\subseteq S$ if $k=2r+2$ or $\{w,w',w''\}\subseteq S$ if $k=2r+3$.

\vspace{3mm}\noindent{\bf Claim 1.} \  $|S\cap V(K_{n-k})|\leq 1$.

\vspace{3mm} By contradiction, suppose that $|S\cap V(K_{n-k})|=2$
and $v_l,v_j\in S$. If $|S\cap\{v_1,\ldots,v_r\}|\leq 1$, there
exists $1\leq i\leq r$ such that $v_i,v_{i+1}\not\in S$ since $r\geq
3$, then $\{w_i,w_i'\}\subseteq  S$. But now,
$\{w_i,v_l,v_j\}\subseteq N(v_i)\cap S$. This contradicts to the
fact that $S$ is a $[1,2]$-set. So we may assume $|S\cap
\{v_1,\ldots, v_r\}|=2$. Then $w\in S$ since $\{v_i\ |\ r+1\leq
i\leq n-k\}\cap S=\emptyset$, which means $\{w,v_l,v_j\}\subseteq
N(v_{r+1})\cap S$, a contradiction.

\vspace{3mm}\noindent{\bf Claim 2.}\ $S\cap V(K_{n-k})=\emptyset$.

\vspace{3mm} By contradiction, suppose that $|S\cap V(K_{n-k})|= 1$
and $v_l\in S\cap V(K_{n-k})$. If $1\leq l\leq r$, then
$\{w,w'\}\subseteq S$, and then $\{w,w',v_l\}\subseteq
N(v_{r+1})\cap S$ if $k=2r+2$ or $\{w,w'',v_l\}\subseteq
N(v_{r+3})\cap S$ if $k=2r+3$, both contradicting to the fact that
$S$ is a $[1,2]$-set. If $r+1\leq l\leq n-k$, then
$\{w_1,w_1',w_2,w_2'\}\subseteq S$, so $\{w_1',w_2,v_l\}\subseteq
N(v_2)\cap S$, a contradiction.

It follows immediately that $S=S'$ and thus $\gamma_{t[1,2]}(H_{n,k})=k$.
\end{proof}

\section{ Graphs with no total $[1,2]$-set}
As we have seen in the previous section, there exist many graphs
with no total $[1,2]$-set, we summarize it as follows.

\begin{theorem}{\label{31}} Let $G$ be a connected graph of order $n$.

(1) If  $n\geq 3$ and $\gamma_{[1,2]}(G)> \frac{4n}{5}$, then
$\gamma_{t[1,2]}(G)=+\infty$;

(2) If  $n\geq 5$, $\delta(G)\geq 2$ and $\gamma_{[1,2]}(G)>
\frac{2n}{3}$, then $\gamma_{t[1,2]}(G)=+\infty$;

(3) Let $G$ be a tree of order $n$ with $k$ leaves. If
$\gamma_{[1,2]}(G)=n-k$, then $\gamma_{t[1,2]}(G)=+\infty$ unless
$G$ is a caterpillar.

\end{theorem}

\begin{proof} (1) and (2) are immediate from Theorem \ref{21} and \ref{22}.

(3) Assume that $\gamma_{t[1,2]}(T)$ exists and $S$ is a
$\gamma_{t[1,2]}(T)$-set of $T$. Let $L$ be the set of all leaves of
$T$. If $v\in L\cap S$, the support vertex of $v$ also lies in $S$.
It follows that $S'=S\setminus L$ is also a $[1,2]$-set of $T$ and
hence $|S'|\leq n-k$. Since $\gamma_{[1,2]}(T)=n-k$, we have
$|S'|=n-k$ and $V(T)\setminus S'=L$. If there exists a vertex $v\in
S'$ such that $N(v)\subseteq S'$, then $S'\setminus \{v\}$ is
 a $[1,2]$-set of $T$ with cardinality less than $n-k$, a contradiction.
 So each vertex in $S'$ must have at least one neighbor outside $S'$, which means all the vertices of $S'$ are support vertices of $T$.
It is easy to see $\Delta(T[S'])\leq \Delta(T[S])\leq 2$. So $T[S']$
must be a path and then $T$ is a caterpillar.
\end{proof}

Let $p$ and $k$ be two integers with $p\geq k+2\geq 5$, $G_{p,k}$ is
a graph obtained from a complete graph $K_p$ as follows: for every
$k$-element subsets $S$ of the vertices set $V(K_p)$, we add a new
vertex $x_s$ and the edges $x_su$ for all $u\in S$. In \cite{wu},
Yang and Wu proved that $\gamma_{[1,2]}(G_{p,k})=|V(G_{p,k})|$. So,
it is immediate from Theorem \ref{31} that

\begin{corollary}
If $p\geq k+2$ and $k\geq 3$, then
$\gamma_{t[1,2]}(G_{p,k})=+\infty$.
\end{corollary}

\section{Graphs with $\gamma_{t[1,2]}(G)=\gamma_{[1,2]}(G)$ }

It was shown in \cite{che} that if $G$ is the corona $H\circ K_1$ of
a graph $H$, then $\gamma_{[1,2]}(G)=\gamma(G)$. For total $[1,2]$-domination number, we proved the following result.

\begin{theorem}
Let $G$ be the corona $H\circ K_1$ of a connected graph $H$ of order
$n\geq 2$. Then $\gamma_{t[1,2]}(G)=\gamma_{[1,2]}(G)$ if and only
if $H$ is a path or a cycle.
\end{theorem}
\begin{proof} To show the sufficiency,
let $S$ be a $\gamma_{t[1,2]}(G)$-set of $G$. Since $S$ is a
dominating set, for a leaf of $G$, either itself is contained in $S$
or its support vertex contained in $S$. Moreover, By the definition
of corona, for two leaves, their support vertices are different. So,
$|S|\geq n$. On the other hand, since $H$ is a path or cycle, $V(H)$
is a total $[1,2]$-set of $G$, $|S|\leq |V(H)|=n$. This proves the
sufficiency.

To show the necessity, we assume that $\gamma_{t[1,2]}(H\circ
K_1)=\gamma_{[1,2]}(H\circ K_1)$, and let $S$ be a
$\gamma_{t[1,2]}(G)$-set of $G$. Obviously, all the support vertices
must lie in $S$, which means $V(H)\subseteq S$. On the other hand,
$\Delta(G[S])\leq 2$, implying that $\Delta(H)\leq 2$. Moreover,
since $H$ is connected, $H\cong P_n$ or $H\cong C_n$.

\end{proof}

\begin{theorem}(\cite{bollo})\label{52}
If $G$ is a graph without isolated vertices, then $G$ has a minimum
dominating set $D$ such that, for all $v\in D$, $pri(v, D)\cap
(V(G)\setminus D)\neq \emptyset$.
\end{theorem}
Using the above result, Chellali et al. \cite{che} gave the
following sufficient condition for a graph $G$ satisfying
$\gamma(G)=\gamma_{[1,2]}(G)$.

\begin{theorem}(\cite{che})\label{53}
If $G$ is a $P_4$-free graph, then $\gamma(G)=\gamma_{[1,2]}(G)$.
\end{theorem}

By slightly simplying the proof of the above theorem, we get the
following stronger result.

\begin{theorem}\label{54}
Let $G$ be a $P_4$-free graph without isolated vertices. If $D$ is a
minimum dominating set of $G$ such that $pri(v, D)\cap
(V(G)\setminus D)\neq \emptyset$  for all $v\in D$, then $D$ is also
a $[1,2]$-set of $G$.
\end{theorem}
\begin{proof} By contradiction, suppose that $D$ is not a $[1,2]$-set of $G$.
Then there exists a vertex $u\in V\setminus D$ with at least three
neighbors $x, y, z$ say, in $D$. For convenience, we use $A_v$ to
denote the set of private neighbors of $v$ in $V\setminus D$ for a
vertex $v\in D$. Hence, $A_v\neq \emptyset$ for every $v\in \{x, y,
z\}$. We consider two cases.

\vspace{3mm}\noindent{\bf Case 1.} $A_x\cup A_y\cup A_z\not\subseteq
N(u)$.

Without loss of generality, let $z'\in A_z\setminus N(u)$. If
$xz\notin E(G)$, then $\{z', z, u, x\}$ induces a $P_4$ in $G$. So,
$xz\in E(G)$. Similarly, we have $yz\in E(G)$. Furthermore, $xy\in
E(G)$, since otherwise, $\{x', x, z, y\}$ induces a $P_4$ in $G$,
where $x'\in A_x$. To avoid a $P_4$ induced by $\{z', z, x, x'\}$
induces a $P_4$ in $G$, we have $x'z'\in E(G)$. But, $\{y, x, x',
z'\}$ will induces a $P_4$ in $G$, a contradiction.

\vspace{3mm}\noindent{\bf Case 2.} $A_x\cup A_y\cup A_z\subseteq
N(z)$.

Note that $(N(y)\cap N(z))\setminus (N(u)\cup N(x))\neq \emptyset$.
If this is not, then $(D\setminus \{y,z\})\cup \{u\}$ is a
dominating set of $G$ with cardinality less than $|D|$, a
contradiction. By a similar argument to the above, we have
$(N(x)\cap N(y))\setminus (N(u)\cup N(z))\neq \emptyset$. Take $p\in
(N(y)\cap N(z))\setminus (N(u)\cup N(x))$ and $q\in (N(x)\cap
N(y))\setminus (N(u)\cup N(z))$ respectively.

If $pq\in E(G)$, then each of $\{u,z,p,q\}$ induces $P_4$ in $G$.
So, $pq\notin E(G)$. If $xy\notin E(G)$, then $pyqx$ induces a
$P_4$; if $yz\notin E(G)$, then $qypz$ induces a $P_4$ in $G$. So,
it follows that $xy\in E(G)$ and $yz\in E(G)$. By a similar argument
to the above, one has $xz\in E(G)$. But, then $\{q,x,z,p\}$ induces
$P_4$, a contradiction.

Consequently, $D$ is a $[1,2]$-set and the theorem is proved.
\end{proof}

\begin{theorem}
If $G$ is a connected $P_4$-free graph of order $n\geq 4$, then
$$\gamma_{t[1,2]}(G)=\left \{
\begin{array}{ll}
2, & \mbox{if $\Delta(G)=n-1$} \\
\gamma_{[1,2]}(G)=\gamma(G), & \mbox{if $\Delta(G)<n-1$.} \\
\end{array}
\right.
$$
\end{theorem}

\begin{proof}If $\Delta(G)=n-1$, then
the result trivially holds. So, next we assume that $\Delta(G)<n-1$.
By Theorem 4.4, let $D$ be a $\gamma_{[1,2]}(G)$-set of $G$ such
that $pri(u, D)\cap (V(G)\setminus D)\neq \emptyset$ for all $u\in
D$. Among all such dominating sets, we choose one such that
$\omega(G[D])$ is as small as possible. We show that $D$ is also a
$\gamma_{t[1,2]}(G)$-set of $G$. By contradiction, suppose that $D$
is not a total $[1,2]$-set. Then, either $\delta(G[D])=0$ or
$\Delta(G[D])\geq 3$.

\vspace{3mm}\noindent{\bf Case 1.} $\Delta([D])\geq 3$.

Suppose $v\in D$ is a vertex having three neighbors $v_1, v_2, v_3$
say, in $D$, and let $w\in pri(v, D)$. $w_i\in pri(v_i, D)$ for
$i=1, 2, 3$. To avoid having a $P_4$ induced by $\{w_1,v_1,v,v_2\}$,
we have $v_1v_2\in E(G)$. Similarly, $ww_2\in E(G)$, since otherwise
$\{w,v,v_2,w_2\} $ induces a $P_4$. But then $\{w,w_2,v_2,v_1\}$
induces a $P_4$, a contradiction.

\vspace{3mm}\noindent{\bf Case 2.} $\Delta(G[D])\leq 2$ and
$\delta(G[D])=0$.

Since $\Delta([D])\leq 2$, all components of $G[D]$ are paths or
cycles. Combing this with the fact that $G$ is $P_4$-free, each
component of $G[D]$ is isomorphic to $K_1, K_2, P_3, K_3$ or $C_4$.
Without loss of generality, let $H_i$ be an isolated vertex for any
integer $i\leq k$ and $H_j$ is a path of order at least two or a
cycle for any $k+1\leq j\leq t$. Since $G$ is $P_4$-free,
$diam(G)\leq 2$, which implies that $N(x)\cap N(y)\neq \emptyset$ if
$x$ and $y$ lie two different components of $G[D]$. Without loss of
generality, let $V(H_1)=\{v_1\}$. Let $v_1wv_j$ be a path connecting
$v_1$ and a vertex $v_j\in V(H_j)$ for some $j\geq 2$. To avoid
having a $P_4$ induced by $\{v_1, w, v_j, v_j'\}$ for any $v_j'\in
N(v_j)\cap V(H_j)$, and thus $V(H_j)\subseteq N(v_1)$. But then
$D'=(D\setminus V(H_j))\cup \{w\}$ is a dominating set of $G$ with
cardinality less than $|D|$ if $|V(H_j)|\geq 2$. So, $|V(H_j)|=1$
for any $j$, i.e. each component of $G[D]$ is an isolated vertex.

Since $\Delta(G)<n-1$, $\gamma_{t[1,2]}(G)=k\geq 2$. Let $D=\{v_1,
v_2,\cdots, v_k\}$. Since $G$ is connected, $N(v_1)\cap N(v_j)\neq
\emptyset$ for some integer $j\in\{2, \ldots, k\}$. Without loss of
generality, let $j=2$. Take $w\in N(v_1)\cap N(v_2)$ and $w_i\in
pri(v_i, D)$ for $i\in\{1, 2\}$. To avoid a $P_4$, $pri(v_1, D)\cup
pri(v_2, D)\subseteq N(v)$, since otherwise, if $w_1w\not\in E(G)$,
$\{w_1,v_1,w,v_2\}$ induces a $P_4$. If $(N(v_1)\cap
N(v_2))\setminus \{w\}\subseteq N(w)$, then $(D\setminus
\{v_1,v_2\})\cup \{w\}$ is a dominating set of $G$, a contradiction.
If $(N(v_1)\cap N(v_2))\setminus \{w\}\not\subseteq N(w)$, then
$D'=(D\setminus \{v_2\})\cup \{w\}$ is a dominating set of $G$.
Since $(N(v_1)\cap N(v_2))\setminus \{w\}\not\subseteq N(w)$,
$pri(v_1,D')\neq \emptyset$; $v_2\in pri(w,D')$,
$pri(v_i,D')=pri(v_i,D)$ for $3\leq i\leq k$ and
$\omega(G[D'])<\omega(G[D])$, contradicting the choice of $D$ that
$\omega(G[D])$ is as small as possible.

\end{proof}

\section{Further research}

It is interesting that the total $[1,2]$-domination problem is
concerned with graph partition and factors. Recall that a spanning
subgraph $H$ of $G$ is called a {\it $[a, b]$-factor} if $a\leq
d_H(v)\leq b$. In particular, $H$ is called a {\it $k$-factor} of
$G$ if $H$ is a $k$-regular spanning subgraph of $G$.

\vspace{3mm} \noindent{\bf Conjecture 1.} For any cubic graph $G$ of
order $n$, $\gamma_{t[1,2]}(G)<n$.

\vspace{3mm}The statement of Conjecture 1 is equivalent to that
every cubic graph $G$ has a vertex partition $(S, V(G)\setminus S)$
such that $1\leq \delta(G[S])\leq \Delta(G[S])\leq 2$ and $1\leq
\delta(G-S)\leq \Delta(G-S)\leq 2$.

It is well-known that every regular graph has a $[1,2]$-factor (see
\cite {SU}), and so does a cubic graph. Hence Conjecture 1 asserts
the existence of $[1, 2]$ -factor with an additional property in a
cubic graph. A theorem of Petersen \cite{P} says that every even
regular graph $G$ has a 2-factor. So, we also pose the following
conjecture.

\vspace{3mm}\noindent{\bf Conjecture 2.} For any 4-regular graph $G$
of order $n$, $\gamma_{t[1,2]}(G)<n$.

\vspace{3mm} The middle levels problem, attributed to H\`{a}vel
\cite{H}, concerns the following family of graphs. Let $n, a, b$ be
integers with $0\leq a< b\leq n$. Let $G(n; a, b)$ denote the
bipartite graph whose vertices are all the $a$-element and
$b$-element subsets of an $n$-set, say $[n] =\{1, 2, \ldots, n\}$.
An $a$-element subset $A$ and a $b$-element subset $B$ are adjacent
in $G(n; a, b)$ if and only if $A\subseteq B$. So, the order of
$G(n; a, b)$ is $\binom {n} a+ \binom {n}b$. The middle levels
problem asks that for a positive integer $k$, is the graph $G(2k+1;
k, k+1)$ Hamiltonian?  It is so-named because it deals with the
central levels of the Boolean algebra $2^{[2k+1]}$. Yang and Wu
\cite{wu} proved the following theorem.

\begin{theorem} (\cite{wu}) Let $n$ and $k$ be
two integers with $n\geq k\geq 3$. If $n$ is sufficiently large with
respect to any fixed $k$, then $\gamma_{[1,2]}(G(n; k, n-k))=|V(G(n;
k, n-k))|$.
\end{theorem}

So, by Theorems 3.1 and 5.1, if $n$ is sufficiently large with
respect to any fixed $k$, then $\gamma_{t[1,2]}(G(n; k,
n-k))=+\infty$. Note that $G(n; k, n-k)$ is a $\binom
{n-k}{k}$-regular bipartite graph of order $2\binom {n}{k}$. So, we
ask the following problem.

\vspace{3mm}\noindent{\bf Question.} What is the smallest integer
$k$ such that there exists a $k$-regular graph $G$ with
$\gamma_{t[1,2]}(G)=+\infty$ ?

Chellali et al. \cite{che} asked that if $\gamma_{[1,2]}(G)<n $ for
any graph 5-regular graph $G$ ? If this is true, it is equivalent to
the following conjecture.

\vspace{3mm}\noindent{\bf Conjecture 3.} Every 5-regular graph $G$
has an induced subgraph $H$ with $3\leq d_H(v)\leq 4$.

\vspace{3mm} Note that a stronger assertion, that every 5-regular
graph $G$ has an 3-regular or 4-regular induced subgraph $H$,
generally does not hold. The formulation of Conjecture 3 reminds us
the well-known Berge-Saure conjecture \cite{A}, confirmed by
T\^{a}skinov \cite{T} and Zhang \cite{Z}, which says that every
4-regular graph contains a 3-regular subgraph.

\end{document}